\documentclass[reqno0,12pt]{amsart}
\pdfoutput=1
\usepackage{bbm}
\usepackage[nodate]{datetime}
\usepackage[hmargin=35mm,top=30mm,bottom=30mm,a4paper]{geometry}
\usepackage{color}
\usepackage{fancyhdr}
\usepackage{amsmath,amssymb,amsfonts,amsthm}
\usepackage{amstext}
\usepackage{amsmath}
\usepackage{amssymb}
\usepackage{enumerate}
\usepackage{amsbsy}
\usepackage{amsopn}
\usepackage{bbm,amsthm}
\usepackage{amscd}
\usepackage{amsxtra}
\usepackage{enumitem}
\usepackage{hyperref}
\usepackage{todonotes}

\newtheorem{theorem}{Theorem}[section]
\newtheorem*{theorem*}{Theorem}
\newtheorem{lemma}{Lemma}[section]

\newtheorem{corollary}{Corollary}[section]

\theoremstyle{remark}

\newtheorem{example}{Example}[section]
\newtheorem{remark}{Remark}[section]

\newcommand{\NN}{\mathbb{N}}

\newcommand{\ind}{\mathbbm{1}}

\newcommand{\EE}{\mathbb{E}}
\newcommand{\PP}{\mathbb{P}}

\begin{document}
\title{Sign changes of the partial sums of a random multiplicative function III: Average}
\author{Marco Aymone}
\begin{abstract}
Let $V(x)$ be the number of sign changes of the partial sums up to $x$, say $M_f(x)$, of a Rademacher random multiplicative function $f$. We prove that the averaged value of $V(x)$ is at least $\gg (\log x)(\log\log x)^{-1/2-\epsilon}$. Our new method applies for the counting of sign changes of the partial sums of a system of orthogonal random variables having variance $1$ under additional hypothesis on the moments of these partial sums. In particular, we extend to larger classes of dependencies an old result of Erd\H{o}s and Hunt on sign changes of partial sums of i.i.d. random variables. In the arithmetic case, the main input in our method is the ``\textit{linearity}'' phase in $1\leq q\leq 1.9$ of the quantity $\log \EE |M_f(x)|^q$, provided by the Harper's \textit{better than squareroot cancellation} phenomenon for small moments of $M_f(x)$. 
\end{abstract}
\maketitle
\section{Introduction}
\subsection{Main result and background} 
In a recent paper by Heap, Zhao and the author \cite{aymone_sign_changes} it was exhibited a proof for the claim that the arithmetic random walk performed by the partial sums of a Rademacher random multiplicative function $f$, say $M_f(x):=\sum_{n\leq x}f(n)$, visits the origin an infinite number of times, almost surely (see also an extension of this result for weighted partial sums, by the author \cite{aymone_sign_changes_II}). The method presented there was used by Geis and Hiary \cite{geis_counting_sign_changes} to give an almost sure lower bound for the number of sign changes $V(x)$ of $M_f(u)$ in the interval $u\in[1,x]$. They showed that
$$V(x)\gg_\delta (\log\log\log x)^{1/2-\delta}, \,a.s., \mbox{ for any }\delta>0.$$

In a very recent update in a preprint by Klurman, Lamzouri and Munsch \cite{klurman_counting_sign_changes}, they showed that
$$V(x)\gg \frac{\log\log x}{\log\log\log\log x}, \,a.s.$$

In this third of a series of papers on sign changes of the partial sums of a random multiplicative function, we are interested in the average of $V(x)$. 

\begin{corollary}\label{corollary average} Let $V(x)$ be the number of sign changes of $M_f(u)$ in the interval $u\in [1,x]$. Then there exists a constant $\kappa>0$, such that for each fixed $0<\epsilon<1/100$, for all $x\geq x_0=x_0(\epsilon)$,
$$\EE V(x)\geq \kappa\frac{\log x}{(\log \log x)^{1/2+\epsilon}}.$$
\end{corollary}

In a series of two papers \cite{pintz_sign_changes_1} and \cite{Pintz_sign_changes_2} by Kaczorowski and Pintz, it was proved the best result known up to date for the Mertens function: There is at least a constant $c>0$ times $\log x$ sign changes of $\sum_{n\leq u}\mu(n)$ in the interval $[1,x]$. Thus our Corollary above shows a similar result between random multiplicative functions with what is known\footnote{For the i.i.d Rademacher random walk, we have, in average, $\asymp\sqrt{x}$ sign changes of this walk up to time $x$. It seems plausible to speculate something similar for the ``M\"obius walk'' and for Rademacher random multiplicative functions.} for the M\"obius function.

\begin{remark} As it is common in \textit{Measure Theory}, the average behavior can be, most part of the times, very different from the almost sure behavior. In this sense, we are unsure whether our Corollary \ref{corollary average} would lead to improvements in the almost sure result of Klurman, Lamzouri and Munsch \cite{klurman_counting_sign_changes}.
\end{remark}

We obtain our Corollary \ref{corollary average} as a consequence of the following local result for sign changes.

\begin{theorem}\label{thorem local sign changes} Let $N\in\NN$ and $x\geq 1$ be sufficiently large and such that, $N=o(\log x)$ and, for fixed $0<\epsilon<1/1000$, $\log \log x\ll  N^{2-\epsilon}$. Then the probability that $M_f(u)$ has at least one sign change in the interval $u\in[x,e^N x]$ is at least an absolute constant $\kappa>0$, for all $x$ sufficiently large.
\end{theorem}

\subsection{A more general Theorem, and examples} Our method has two main inputs. The first one is the orthogonality of the random variables $(f(n))_n$. Under this orthogonality, it is easy to show that the correlation between $M_f(e^nx)$ and $M_f(e^mx)$
is exponentially small on the quantity $|m-n|$.

The second main input is the remarkable Harper's \textit{better than square root cancellation} \cite{harperfirstmoment} for the moments $\EE |M_f(x)|^q$, for $1\leq q<2$. Actually, the better than the square root cancellation is not, a priori, the key feature to prove our results, but the following linearity in $q$:
$$\log \EE |M_f(x)|^q \asymp \frac{q}{2}\log \left(\frac{x}{(1+(1-q/2)\sqrt{\log\log x})}\right).$$ 

Our method allows us to prove the following result.

\begin{theorem}\label{theorem geral} Let $(X_n)_n$ be orthogonal random variables such that $\EE X_n=0$ and $\EE X_n^2=1$ for all $n$. Let $M(u):=\sum_{n\leq u}X_n$ denote the partial sums of $X_n$. Assume that for some $1\leq q_1<q_2\leq 2$, as $x\to\infty$ the $L^{q_1}$ and $L^{q_2}$ norm satisfy the linearity condition $\|M(x)\|_{q_1}\asymp \|M(x)\|_{q_2}$. Assume further that 
$$\|M(x)\|_{q_1}\asymp\frac{\sqrt{x}}{\psi(x)},$$
for some continuous and and non-decreasing function $\psi:[1,\infty)\to[1,\infty)$. Further, assume that for $n=o(\log x)$,
$$\psi(e^n x)=\psi(x)\left(1+O\left(\frac{n}{\log x}\right)\right).$$
If $\psi(x)^{2+\epsilon}\ll N=o(\log x)$ for some small fixed $0<\epsilon<1/100$, then the probability that $M(u)$ has at least one sign change in the interval $u\in[x,e^N x]$ is at least an absolute constant $\kappa$, if $x$ is sufficiently large.
\end{theorem}

\subsubsection{Examples}

\begin{example}
If $\psi(x)$ can be taken to be a constant in Theorem \ref{theorem geral}, then the number of sign changes $V(x)$ of the partial sums $M(u)$ in the interval $u\in[1,x]$ is such that 
$$\EE V(x)\gg \log x.$$

Concrete examples of this can be achieved via \textit{Martingale Theory}  in Probability. Indeed, let $(X_n)_n$ be a Martingale difference, that is, a sequence of random variables with the following properties:
\begin{itemize} 
\item $\EE |X_n|<\infty$, for all $n$;
\item There is a sequence of filtrations $(\mathcal{F}_n)_n$ such that $X_n$ is $\mathcal{F}_n$-measurable;
\item $\EE(X_n|\mathcal{F}_{n-1})=0$, for all $n\geq 1$.
\end{itemize}

The tower property for conditional expectation guarantees that the random variables $(X_n)$ are orthogonal and $\EE X_n=0$, for all $n$. Moreover, we have the Burkholder's inequality (see the book of Shiryaev \cite{shiryaev} pg. 499): For each $q>1$, there are constants $c_q,C_q>0$ such that 
$$c_q\left\|\sqrt{\sum_{n\leq x}X_n^2}\right\|_q\leq \|M(x)\|_q\leq C_q\left\|\sqrt{\sum_{n\leq x}X_n^2}\right\|_q.$$

So if we further assume that for some positive constants $A$ and $B$, $A\leq|X_n|\leq B$ for all $n$, then the hypothesis of Theorem \ref{theorem geral} are satisfied with $\psi(x)=constant$. Other subset of hypothesis on $(X_n)_n$ could lead to same conclusions.
\end{example}

\begin{example} Theorem 1 of Erd\H{o}s and Hunt \cite{erdos_hunt_sign_changes} states that if $(X_n)_n$ are symmetric and i.i.d. random variables, then
$$ \EE V(x)\geq \frac{1}{2}\log x +O(1).$$

Our Theorem \ref{theorem geral}, in particular, partially recovers this result of Erd\H{o}s and Hunt. We say partially, because in \cite{erdos_hunt_sign_changes} they do not need any other assumption on the moments of $(X_n)_n$, while our Theorem \ref{theorem geral} demands the existence of two moments that match their order of magnitude. On the other hand, we do not need the symmetry hypothesis, and as our example before shows, our method can be used in another settings of dependencies among $(X_n)_n$. 

Therefore, if in addition we assume that $(X_n)_n$ are i.i.d. and $\EE X_1^{2+\epsilon}<\infty$, then we can use the Burkholder's inequality (or the Marcinkiewicz-Zygmund's inequality) of the example before to guarantee that the first moment has the same order of magnitude of the second moment, via an interpolation-norm argument. With this we conclude that 
$$\EE V(x)\gg \log x.$$

\end{example}

\begin{example}[$\mathcal{B}_2$ or Sidon sets] A set $S=\{n_1,n_2,...\}$ of positive integers is called $\mathcal{B}_2$ or Sidon if there exists a constant $C>0$ such that for all $n\geq 1$, the equation
$$n=n_j\pm n_k$$
has at most $C$ solutions with $n_j$ and $n_k$ in $S$. By letting $U$ a random variable with uniform distribution over $[0,2\pi]$ and $M(x)$ the cosine polynomial
$$M(x):=\sum_{k\leq x}\cos(n_k U),$$
an old result of Sidon \cite{sidon} states that 
$$\EE M(x)^4\asymp (\EE M(x)^2)^2\asymp x^2.$$
This implies that
$$\EE |M(x)|\asymp (\EE M(x)^2)^{1/2},$$
and hence the hypothesis of Theorem \ref{theorem geral} are satisfied with $\psi(x)$ a constant. Thus, in this case we also have that
$$\EE V(x)\gg \log x.$$

\end{example} 

\begin{example} Our method is flexible enough to hold other situations where not necessarily $\EE X_n^2\gg 1$. For instance, the interested reader can outline the details and prove, by using our method, that in the case $X_n=r_n/\sqrt{n}$, where $(r_n)_n$ are i.i.d. Rademacher random variables, that 
$$\EE V(x)\gg \log\log x.$$
\end{example}

\section{Preliminaries}
\subsection{Notation} 
In this subsection we make a summary of all recurrent notation used in this paper. We hope that it may be useful for the reader always when he or she becomes overloaded with the plenty number of notation used here.

\subsubsection{Letters appearing throughout the text} We will let the letter $p$ to always represent a generic prime number, $n$ to represent a positive integer, $x$ and $y$ real variables used as the edge of an index of summation. $q$ is reserved for the size of the moments of certain random variables. We let $c_j$ to represent positive auxiliary constants and greek letters for the ultimate constants. $\lambda$ will be used as a threshold in a event where a random variable is above or beyond this threshold. The letter $f$ represents our Rademacher random multiplicative function. $\PP$ is the probability of an event.

\subsubsection{Asymptotic notation} We use the standard Vinogradov notation $f(x)\ll g(x)$ or Landau's $f(x)=O(g(x))$ whenever there exists a constant $c>0$ such that $|f(x)|\leq c|g(x)|$, for all $x$ in a set of parameters. When not specified, this set of parameters is an infinite interval $(a,\infty)$ for sufficiently large $a>0$. Sometimes is convenient to indicate the dependence of this constant in other parameters. For this, we use both $\ll_\delta$ or $O_\delta$ to indicate that $c$ may depends on $\delta$. We say that $f\asymp g$ if both $f\ll g$ and $g\ll f$ are realized. The standard $f(x)=o(g(x))$ means that $f(x)/g(x)\to0$ when $x\to a$, where $a$ could be a complex number or $\pm \infty$. 
\subsection{Some estimates}
For $x\geq 1$, $N\in\NN$, and $1\leq q<2$, define
\begin{equation}\label{equacao definicao Lambda}
\Lambda(N,x,q)=\sum_{n\leq N}\frac{1}{(1+(1-q/2)\sqrt{\log\log(e^nx)})^{q/2}}.
\end{equation}
\begin{lemma}\label{lemma estimate for Lambda} For $\Lambda(N,x,q)$ defined as above, if $1\leq q<1.9$, $N=o(\log x)$, then, as $x,N\to\infty$
$$\Lambda(N,x,q)= \frac{1+o_q(1)}{(1-q/2)^{q/2}}\frac{N}{(\log\log x)^{q/4}}.$$
\end{lemma}
\begin{proof}
By doing some routine Taylor expansion, we have that, if $N=o(\log x)$,
$$\log\log (e^n x)=\log\log x+O\left(\frac{N}{\log x}\right).$$
This allows us to infer that
\begin{multline*}
1+(1-q/2)\sqrt{\log\log(e^nx)}\\=(1-q/2)\sqrt{\log\log x}\left(1+\frac{1}{(1-q/2)\sqrt{\log\log x}}+O\left(\frac{N}{(\log x)(\log\log x)}\right)\right).\\
\end{multline*}
On the other hand,
\begin{multline*}
\frac{1}{1+\frac{1}{(1-q/2)\sqrt{\log\log x}}+O\left(\frac{N}{(\log x)(\log\log x)}\right)}\\=1+O\left( \frac{1}{(1-q/2)\sqrt{\log\log x}}\right)+O\left(\frac{N}{(\log x)(\log\log x)}\right).
\end{multline*}
Finally, by combining these estimates, we obtain that
$$\frac{1}{(1+(1-q/2)\sqrt{\log\log(e^nx)})^{q/2}}=\frac{1}{(1-q/2)^{q/2}(\log\log x)^{q/4}}(1+o_q(1)),$$
and summing over $n$ the expression above, we obtain the desired estimate. \end{proof}

\section{Proof of the main result}
\subsection{Setting the definitions} Let $x>0$ be large and set 
\begin{equation}\label{equation definition Y}
Y_n=\frac{M_f(e^nx)}{\sqrt{e^n x}}.
\end{equation}

Since the quantity of squarefree numbers has density $\frac{6}{\pi^2}$, we have that for sufficiently large $x$, $\frac{3}{\pi^2}\leq\EE Y_n^2\leq 1$, for all $n$.
We define 
\begin{align}\label{equation difinition S and S*}
S_N&:=Y_1+...+Y_N,\\
S_N^{*}&:=|Y_1|+...+|Y_N|.
\end{align} 
 
\begin{lemma}\label{lema correlacao} For $x$ sufficiently large, the correlation $\rho_{n,m}$ between $Y_n$ and $Y_m$ for $n<m$ has absolute value at most 
$$ \frac{c_1}{e^{(m-n)/2}},$$
for some positive constant $c_1$.
\end{lemma} 
\begin{proof}
The correlation between $Y_n$ and $Y_m$ is given by
$$\frac{\EE Y_nY_m-(\EE Y_n)(\EE Y_m)}{\sqrt{\EE Y_n^2-(\EE Y_n)^2}\sqrt{\EE Y_m^2-(\EE Y_n)^2}}.$$
By orthogonality of $(f(n))_n$,
$$\EE Y_nY_m=\frac{\EE M_f(e^n x)^2}{e^{(n+m)/2}x}=(6/\pi^2+o(1))\frac{e^n}{e^{(n+m)/2}}=(6/\pi^2+o(1))\frac{1}{e^{(m-n)/2}}.$$
Now
$$(\EE Y_n)(\EE Y_m)=\frac{1}{e^{(n+m)/2}}.$$
And moreover, for sufficiently large $x$, for all $n,m\geq 1$
$$\EE Y_n^2-(\EE Y_n)^2\geq \frac{6}{\pi^2}-o(1)-e^{-2}\geq 0.4-o(1),$$
and hence the claim follows.\end{proof}

\begin{lemma}\label{lemma chebyshev bound} Let $S_N$ be as above and $x$ be sufficiently large. Then, there exists a constant $c_2>0$ such that, for any $\lambda>0$
$$\PP(|S_N|\geq \lambda)\leq \frac{c_2N}{\lambda^2}.$$
\end{lemma}
\begin{proof}
We have that, for $x$ sufficiently large, $\EE Y_n^2\leq 1$ and Lemma \ref{lema correlacao} is applicable so that
\begin{align*}
\EE S_N^2\leq N+2\sum_{1\leq n\leq N-1}e^{n/2}\sum_{n<m\leq N}e^{-m/2}\ll N.
\end{align*}
The claim follows by the Markov's inequality.
\end{proof}

\subsection{Proof of Theorem \ref{thorem local sign changes}}
\begin{proof} Let $S_N^*$ be as in \eqref{equation difinition S and S*} above. Let $N=o(\log x)$. We have that, by Harper's result on small moments (Theorem 2 of \cite{harperfirstmoment})
$$\EE S_N^*\asymp\Lambda(N,x,1).$$
Now, for some constant $c_3>0$ and $0<\epsilon<c_3$
$$c_3\Lambda(N,x,1)\leq\EE (S_N^*\ind_{S_N^*< \epsilon\Lambda(N,x,1)})+\EE (S_N^*\ind_{S_N^*\geq \epsilon\Lambda(N,x,1)}).$$
Therefore
\begin{align*}
(c_3-\epsilon)\Lambda(N,x,1)&\leq \EE (S_N^*\ind_{S_N^*\geq \epsilon\Lambda(N,x,1)})\\
&\leq (\EE (S_N^{*3/2}))^{2/3}(\PP(S_N^*\geq \epsilon\Lambda(N,x,1)))^{1/3}.
\end{align*}
Now we have that
$$\EE (S_N^{*3/2})\leq \sqrt{N}\EE\sum_{n\leq N}|Y_n|^{3/2}\leq c_4\sqrt{N}\Lambda(N,x,3/2).$$
Our target is to make $x$ arbitrarily large and keep $N=o(\log x)$. In particular, Lemma \ref{lemma estimate for Lambda} is applicable and hence
$$\EE (S_N^{*3/2})\leq c_5 \sqrt{N}\frac{N}{(\log\log x)^{3/8}}.$$
On the other hand, 
$$\Lambda(N,x,1)\geq c_6\frac{N}{(\log\log x)^{1/4}}.$$

This allow us to infer that that for some constant $c_7>0$ depending in $\epsilon$, the probability
$$\PP(S_N^*\geq \epsilon\Lambda(N,x,1))\geq c_7,$$
for all $x$ sufficiently large and all $N=o(\log x)$.

To complete the argument, we see that by Lemma \ref{lemma chebyshev bound}, for any fixed and small $\delta>0$
$$\PP(|S_N|\geq \Lambda(N,x,1)^{1-\delta})\ll\frac{N}{N^{2-2\delta}}(\log\log x)^{(1-\delta)/2}.$$
As long $\log\log x \ll N^{2-\epsilon}$ where $\epsilon=4\delta/(1-\delta)$, the probability above is $O(N^{-\delta})$, in our range of parameters $x$ and $N$.

Finally, observe that under the events
\begin{align*}
&A:=[S_N^*\geq \epsilon\Lambda(N,x,1)],\\
&B:=[|S_N|\leq \Lambda(N,x,1)^{1-\delta}],
\end{align*}
we must have at least one sign change in the sequence $Y_1,...,Y_N$, meaning that $M_f(y)$ has at least one sign change in the interval $[x,e^Nx]$. To finish the proof, since $\PP(B)=1-o(1)$ and $\PP(A)\geq c_7$, we have that $\PP(A\cap B)\geq c_7/2:=\kappa>0$, for all $x$ and $N$ sufficiently large and in the claimed range. \end{proof}

\subsection{Proof of Corollary \ref{corollary average} }

\begin{proof}
We let $\ell$ run over the positive integers and choose $x_\ell=e^{\ell(\log \ell)^{1/2+\epsilon}}$, for some fixed and small $0<\epsilon<1/100$. Define
$$e^N:=\frac{x_{\ell+1}}{x_\ell}=e^{(\log \ell)^{1/2+\epsilon}+O(1)}.$$
Then $N=o(\log x_\ell)$ and $\log\log x_\ell=\log \ell+(1/2+\epsilon)\log\log \ell\ll N^{2-\epsilon/2}$, and hence Theorem \ref{thorem local sign changes} is applicable:
The probability that $M_{f}(u)$ has at least one sign change in the interval $u\in [x_\ell,x_{\ell+1}]$ is lower bounded by some absolute $\kappa>0$, for all $\ell\geq \ell_0$.

Denote by $V(a,b)$ the number of sign changes of $M_f(u)$ in the interval $u\in(a,b]$.
Then
$$V(x)\geq\sum_{x_\ell\leq x}V(x_{\ell-1},x_\ell).$$
Now, $V(x_{\ell-1},x_\ell)$ is lower bounded by $\ind_{V(x_{\ell-1},x_\ell)\geq1}$, since it is a non-negative and integer valued random variable.
Thus, if $\ell^*$ is the largest integer such $x_{\ell^*}\leq x$,
$$\EE V(x)\geq \kappa \sum_{x_{\ell_0}\leq x_\ell\leq x}1\geq \kappa \ell^*-\kappa \ell_0.$$
To finish the argument, as $x\to\infty$ we have that $\ell^*=(1+o(1))\frac{\log x}{(\log\log x)^{1/2+\epsilon}}$, and this completes the proof.
\end{proof}

\section{Proof of the General Theorem}
Now we turn our attention to
\begin{proof}[Proof of Theorem \ref{theorem geral}] Without loss of generality, we may assume that $q_1=1$. Otherwise, we interpolate the norms and reach that 

$$\|M(x)\|_1\asymp \frac{\sqrt{x}}{\psi(x)}\asymp \|M(x)\|_{q_2}.$$

Similarly as in Lemma \ref{lemma estimate for Lambda}, define for $j=1,2$, 
$$\Lambda_j(N,x)=\sum_{n\leq N}\frac{1}{\psi(e^n x)^{q_j}},$$
where $q_1=1$.

The hypothesis
$$\psi(e^n x)=\psi(x)\left(1+O\left(\frac{n}{\log x}\right)\right)$$
allows us to infer that, as $x\to\infty$ with $N=o(\log x)$
$$\Lambda_j(N,x)=(1+o(1))\frac{N}{\psi(x)^{q_j}}.$$

Let $Y_n=M(e^n x)/(\sqrt{e^n x})$ and
$$S_N:=Y_1+...+Y_N,$$
$$S_N^*:=|Y_1|+...+|Y_N|.$$

Similarly as in Lemma \ref{lemma chebyshev bound}, the orthogonality and unit variance of $(X_n)_n$ allow us to infer that, 
$$\PP(|S_N|\geq \Lambda_1(N,x)^{1-\delta})\leq \frac{\psi(x)^{2(1-\delta)}}{N^{1-2\delta}}.$$
If $\psi(x)^2\ll N^{1-3\delta}$, the probability above is $O(N^{-\delta})$.

Now, repeating the argument as before, for sufficiently small $\epsilon>0$ we have that

$$\frac{N}{\psi(x)}\ll \EE S_N^*\ind_{S_N^*\geq \epsilon\Lambda_1(N,x)}\leq \|S_N^*\|_{q_2}(\PP(S_N^*\geq \epsilon\Lambda_1(N,x)))^{1-1/q_2}.$$

By the Minkowski inequality for the $L^{q_2}$ space, we have that
$$\|S_N^*\|_{q_2}\leq \sum_{n\leq N}\|Y_n\|_{q_2}\ll \frac{N}{\psi(x)}.$$
Thus, for all $x$ sufficiently large, the probability
$$\PP(S_N^*\geq \epsilon\Lambda_1(N,x))\geq \kappa,$$
for some $\kappa>0$.

As before, under the events $|[S_N|\leq \Lambda_1(N,x)^{1-\delta]}$ and $[S_N^*\geq \epsilon\Lambda_1(N,x))]$, if $x$ and $N$ are large enough, we have that $M(u)$ has at least one sign change in the interval $[x,e^N x]$. Moreover, the intersection of this two events has probability at least
$$\kappa+O(N^{-\delta}).$$

Thus, the probability to have at least one sign change in the interval $[x,e^N x]$ is at least $\kappa/2$, if $x$ and $N$ are sufficiently large.

\end{proof}

\noindent\textbf{Acknowledgements.} I would like to thank Winston Heap for the many helpful comments, Caio Bueno for a careful revision of a draft version of this paper, and to Bryce Kerr for pointing out the example of Sidon sets. I also am thankful to the anonymous referee for his/her suggestions that improved the exposition of this paper. The research of the author is funded by FAPEMIG grant Universal APQ-00256-23 and by CNPq grant Universal 403037/2021-2.   

\noindent\textbf{Ethical statement.} The author declares that he followed the ethical rules of this journal and its publisher in the preparation of this manuscript.

\noindent\textbf{Conflict of interest.} The author declares that he has no conflict of interest.

\vspace{2cm}

{\small{\sc \noindent
Departamento de Matem\'atica, Universidade Federal de Minas Gerais (UFMG), Av. Ant\^onio Carlos, 6627, CEP 31270-901, Belo Horizonte, MG, Brazil.} \\
\textit{Email address:} \verb|aymone.marco@gmail.com| }

\end{document}